\documentclass[12pt,oneside,a4paper]{amsart}

\usepackage{amscd}
\usepackage{amssymb}
\usepackage{amstext}
\usepackage{amsthm}
\usepackage{mathrsfs}

\usepackage{latexsym}
\usepackage{xcolor}
\numberwithin{equation}{section}

\DeclareMathOperator{\dist}{dist}

\DeclareMathOperator{\card}{card}

\renewcommand{\phi}{\varphi}

\newtheorem{Thm}{Theorem}
\newtheorem{theorem}[Thm]{Theorem}
\newtheorem{lemma}[Thm]{Lemma}

\newtheorem{proposition}[Thm]{Proposition}

\newtheorem{remark}[Thm]{Remark}

\DeclareMathOperator{\Type}{\tt {Type}}
\DeclareMathOperator{\Cart}{{\tt Cart}}
\DeclareMathOperator{\Pw}{\tt PW_\pi}
\DeclareMathOperator{\Spa}{span}

\begin{document}

\title[Exponential approximation]{Exponential approximation and meromorphic interpolation}

\author{Yurii Belov, Alexander Borichev, Alexander Kuznetsov}
\address{
Yurii Belov:
\newline Department of Mathematics and Computer Science, St.~Petersburg State University, St. Petersburg, Russia
\newline {\tt j\_b\_juri\_belov@mail.ru}
\smallskip
\newline \phantom{x}\,\, Alexander Borichev:
\newline Institut de Math\'ematiques de Marseille,
Aix Marseille Universit\'e, CNRS, I2M, Marseille, France
\newline {\tt alexander.borichev@math.cnrs.fr}
\smallskip
\newline \phantom{x}\,\, Alexander Kuznetsov:
\newline Department of Mathematics and Computer Science, St.~Petersburg State University, St.~Petersburg, Russia
\newline {\tt alkuzn1998@gmail.com}
\smallskip
}
\thanks{The work of the third author was carried out with the financial support of the Ministry of Science and Higher Education of the Russian Federation in the framework of a scientific project under agreement No. 075-15-2024-631. 
} 

\begin{abstract} We establish a relation between the approximation in $L^2[-\pi,\pi]$ by exponentials with the set of frequencies of Beurling--Malliavin density less than $1$ and the meromorphic interpolation 
at $\mathbb Z$. Furthermore, 
we show that typical $L^2[-\pi,\pi]$ functions admit such an approximation.
\end{abstract}

\maketitle

\section{Introduction and main results}
Representations of square integrable functions via exponential series is a classical topic in analysis, see for example \cite{ss1}. For some exciting recent results see \cite{ss2,ss4,ss3}.   

It is well-known that the system of exponentials $\{e^{int}\}_{n\in\mathbb{Z}}$ is an orthonormal basis in $L^2[-\pi,\pi]$. Therefore, to reconstruct a function $f\in L^2[-\pi,\pi]$, 
we need to know $\{\widehat{f}(n)\}_{n\in\mathbb{Z}}$, and all these values are indispensable.

On the other hand, it looks plausible that given a generic function $f\in L^2[-\pi,\pi]$
we may need a smaller amount of exponentials, say with Beurling--Malliavin density strictly less than $1$.

In this paper we introduce a simple probability model and establish such efficient representation for this model. It seems that there is no canonical way to introduce a probability measure on $L^2[-\pi,\pi]$. Nevertheless, the model we consider here looks sufficiently natural. 

Let $\omega\in L^2(\mathbb R_+)$ be a decreasing function such that for some $\delta>0$ we have 
$x\omega(x)\ge \delta$ and $\omega(2x)\ge \delta\omega(x)$ for $x\ge 1$. Let $(\zeta_m)_{m\in\mathbb Z}$ be the sequence of independent standard Gaussian complex variables (complex random variables whose real and imaginary parts are independent normally distributed random variables with mean zero and variance $1/2$). 

We are going to consider {\it random } $f\in L^2[-\pi,\pi]$ defined by
$$
\widehat{f}(n)=\zeta_n \omega(|n|),\quad n\in\mathbb{Z}.
$$

Since $\omega\in L^2(\mathbb R_+)$, we have 
$$
\mathbb E\Bigl(\sum_{n\in\mathbb{Z}}|\widehat{f}(n)|^2\Bigr)=
\mathbb E\Bigl(\sum_{n\in\mathbb{Z}}|\zeta_n|^2\cdot \omega(|n|)^2\Bigr) = 
\sum_{n\in\mathbb{Z}}\omega(|n|)^2<\infty\,.
$$
Therefore, almost surely, $\{\widehat{f}(n)\}_{n\in\mathbb{Z}}\in\ell^2(\mathbb{Z})$, and hence, by the Parseval equality, almost surely, 
$f\in L^2[-\pi,\pi]$. 

\begin{theorem} There exist $\varepsilon>0$ 
such that for almost all random $f\in L^2[-\pi,\pi]$, we can find  
$\Lambda=\Lambda(f)\subset\mathbb R$ such that 
$D_{BM}(\Lambda)<1-\varepsilon$ and  
\begin{equation}
f\in\Spa\{e^{i\lambda t}: \lambda\in\Lambda\}.
\label{mainincl}
\end{equation}
\label{t1}
Moreover, the system $\{e^{i\lambda t}: \lambda\in\Lambda\}$ is a Riesz basis in its closed linear span in $L^2[-\pi,\pi]$, and
$$
f(t)=\sum_{\lambda\in\Lambda} a_\lambda e^{i\lambda t}
$$
with convergence in $L^2[-\pi,\pi]$, for some coefficients $(a_\lambda)_{\lambda\in\Lambda}\in\ell^2$.
\end{theorem}

Here $D_{BM}(\Lambda)$ is the Beurling--Malliavin density of $\Lambda$ (see \cite{BM}), equal to $\inf\bigl\{a:\exists f\in L^2[-a,a]:\mathcal Ff|_\lambda=0\bigr\}$, $\mathcal F$ being the Fourier transform.
For more information on the Beurling--Malliavin density and its relations to other densities see, for example, \cite{RED,KRA,K}. 
In particular, it is known (see, for instance, \cite{Kahane}) that $D_{BM}(\Lambda)$ could be infinite while the linear density of $\Lambda$, 
$$
D(\Lambda)=\lim_{R\rightarrow\infty}\frac{\card(\Lambda\cap[-R,R])}{2R},
$$
is zero. However, in Theorem~\ref{t1} we could require that $D_{BM}(\Lambda)=D(\Lambda)<1-\varepsilon$.

Although the Fock space does not permit even a Riesz basis of reproducing kernels, it looks plausible that a corresponding result should be true there as well for suitable densities, see the techniques introduced in \cite{BBB,BB}. 

Below, in Remark~\ref{rem9}, we show that $\varepsilon$ in Theorem \ref{t1} cannot be taken larger than $\frac{1}{2}$. 
It would be 
of interest to get non-trivial 
estimates on possible $\varepsilon$. 
At this moment, we do not have 
such estimates. 
Furthermore, it is easy to find $f\in L^2[-\pi,\pi]$ which do not permit the representation \eqref{mainincl} with $\Lambda=\mathcal Z(V)$ for real entire functions $V$ in the Cartwright class of exponential type less than $\pi$, with simple real zeros. 
In the proof of Theorem \ref{t1}, we use a non-linear meromorphic interpolation procedure. To prove the probability assertion we first verify a deterministic condition and then use it to get an interpolation result.  

\subsection*{Outline of the paper.} In Section~\ref{sec2} we first establish Proposition~\ref{thm1} that gives a representation of random sequences on $\mathbb Z$ as linear combinations of the Cauchy kernels 
with the set of poles of density smaller than $1$ in the real Gaussian case. Its variant for the complex Gaussian case is Proposition~\ref{thm1c}. 
In Section~\ref{sec3} we start with 
Lemma~\ref{lem22} that shows that such a linear combination of the Cauchy kernels is the quotient of two Cartwright functions. 
Theorem~\ref{thm22} shows that under some natural conditions, approximation in the Paley--Wiener space, by the reproducing kernels, corresponding to the zeros of an entire function of exponential type, is equivalent to an interpolation property. Together, these results give Theorem~\ref{thm8} which contains and somewhat extends our Theorem~\ref{t1}. 
We conclude Section~\ref{sec3} by several remarks.

\section{An interpolation result}
\label{sec2}

We say that $\Lambda\subset\mathbb R$ is separated if $\inf\{|\lambda_1-\lambda_2|:
\lambda_1,\lambda_2\in\Lambda,\,\lambda_1\not=\lambda_2\}>0$.  Given a subset $Q$ of $\mathbb R$ we define its counting function $n_Q$ as 
\begin{equation}
n_Q(t)=\begin{cases}
\card (Q\cap[0,t)),\quad t\ge 0,\\
-\card (Q\cap(t,0)),\quad t< 0,
\end{cases}
\label{d02}
\end{equation}

For two positive functions $f,g$ we say that $f$ is
dominated by $g$, denoted by $f\lesssim g$, if there is a constant
$c>0$ such that $f\le cg$.  

We assume that $\omega:\mathbb R_+\to (0,1]$ is a decreasing function such that for some $\delta>0$ we have 
\begin{align}
x\omega(x)&\ge \delta,\qquad x\ge 1,\label{d01} \\
\omega(2x)&\ge \delta\omega(x),\qquad x\ge 0.\label{d17}
\end{align}

Our key technical result shows that a Gaussian sequence can be interpolated by a sum of the Cauchy kernels with the pole set of density less than $1$. 

\begin{proposition} There exists $\varepsilon>0$ 
such that if 
$\zeta=(\zeta_m)_{m\in\mathbb Z}$ is the sequence of independent standard Gaussian real variables 
{\rm(}with mean zero and variance $1${\rm)},  
then, almost surely, 
there exists a function $F$ meromorphic in the complex plane, whose set of poles $Q=(q_k)_{k\in\mathbb Z}$ is a separated subset of the real line, such that 
$$
F(m)=\zeta_m\omega(|m|),\qquad m\in\mathbb Z,
$$
\begin{equation}
|F(z)|\lesssim 1+\frac1{\dist(z,Q)}, \qquad z\in\mathbb C,
\label{d03}
\end{equation}
and
\begin{equation}
|n_Q(t)-(1-\varepsilon)t|\lesssim 1+|t|^{2/3},\qquad t\in\mathbb R,
\label{lalim}
\end{equation}
with the implicit constants in \eqref{d03} and \eqref{lalim} depending on $\zeta$, $\delta$.
\label{thm1}
\end{proposition}

\begin{proof} 
We extend $\omega(x)=\omega(-x)$ for $x\in\mathbb R_-$. 

Given $A=(A_1,A_2,A_3)\in\mathbb R^3$ we set 
$$
f_A(x)=\frac{A_1}{x+A_2}-\frac{A_1}{x+A_3},\qquad x\in\mathbb R,
$$
and define 
\begin{gather*}
\Gamma=\bigl\{ (A_1,A_2,A_3)\in\mathbb R^3:\{A_2,A_3\}\cap\{-1,0,1\}\ne\emptyset\} \bigr\},\\
L:A\in\mathbb R^3\setminus\Gamma \mapsto L(A)=\bigl(f_A(-1),f_A(0),f_A(1)\bigr)\in\mathbb R^3.
\end{gather*}

We equip $\mathbb R^3$ with the maximum norm and define $D(X,r)=\{Y\in\mathbb R^3:\|Y-X\|<r\}$, $X\in\mathbb R^3$. 

Set $A^*=(\frac18,-\frac12,\frac12)$. Then $LA^*=(\frac16,-\frac12,\frac16)$. The Jacobian $J_L$ of $L$ is invertible at $A^*$,
$$
J_L(A^*)=\begin{pmatrix}
\frac{4}{3} &  -\frac{1}{18} & \frac{1}{2} \\
-4 &  -\frac{1}{2} & \frac{1}{2} \\
\frac{4}{3} &  -\frac{1}{2} & \frac{1}{18}
\end{pmatrix},
$$
$\det JL (A^*)=128/81$, and 
$$
J_{L^{-1}}(LA^*)=\begin{pmatrix}
\frac{9}{64} &  -\frac{5}{32} & \frac{9}{64} \\
\frac{9}{16} &  -\frac{3}{8} & -\frac{27}{16} \\
\frac{27}{16} &  \frac{3}{8} & -\frac{9}{16}
\end{pmatrix}.
$$
Next, we can find $\gamma_1,\gamma_2\in(0,1/4)$ such that 
$L|_{D(A^*,\gamma_1)}$ is a diffeomorphism, 
\begin{equation}
D(LA^*,\gamma_2)\subset L\bigl(D(A^*,\gamma_1)\bigr).
\label{lat}
\end{equation}
We will define $L^{-1}$ on $D(LA^*,\gamma_2)$ with values in $D(A^*,\gamma_1)$. 
Furthermore, for sufficiently small $\gamma_1,\gamma_2$ we have 
\begin{align}
\|L^{-1}A-L^{-1}A'\|&\le 3\|A-A'\|,\qquad A,A'\in D(LA^*,\gamma_2),\label{latt}\\
|f_{A}(z)|&\le \frac{2}{|z|^2+1},\qquad |z|\ge 2,\,A\in D(A^*,\gamma_1),\label{lain}\\
|f_{A}(x)-f_{A'}(x)|&\le \frac{8\|A-A'\|}{x^2+1},\qquad |x|\ge 2,\,A,A'\in D(A^*,\gamma_1).\label{lar}
\end{align}

To verify \eqref{latt}, we just use that $\|J_{L^{-1}}(LA^*)\|_{\mathbb R^3\to \mathbb R^3}<3$.

Next, \eqref{lain} follows from the estimate
\begin{multline*}
|f_{A}(z)|\le \frac{|A_1|\cdot|A_2-A_3|}{|z+A_2|\cdot|z+A_3|}\\ \le \Bigl(\frac18+\gamma_1\Bigr)\frac{1+2\gamma_1}{|z-1/2-\gamma_1|^2} \le \frac{2}{|z|^2+1},\qquad |z|\ge 2.
\end{multline*}

Let now $A=(A_1,A_2,A_3),A'=(A'_1,A'_2,A'_3)\in\mathbb R^3$. To verify \eqref{lar}, we use that 
\begin{multline*}
|f_{A}(x)-f_{A'}(x)|=\Bigl| \frac{A_1}{x+A_2}-\frac{A_1}{x+A_3}-\frac{A'_1}{x+A'_2}+\frac{A'_1}{x+A'_3} \Bigr|\\ \le 
\Bigl| \frac{A_1}{x+A_2}-\frac{A_1}{x+A'_2}\Bigr|+\Bigl|  \frac{A_1}{x+A_3}-\frac{A_1}{x+A'_3} \Bigr|
+|A_1-A'_1|\cdot \Bigl| \frac{1}{x+A'_2}-\frac{1}{x+A'_3}\Bigr|\\
=\frac{|A_1|\cdot |A_2-A'_2|}{|x+A_2|\cdot |x+A'_2|}+\frac{|A_1|\cdot |A_3-A'_3|}{|x+A_3|\cdot |x+A'_3|}+
\frac{|A_1-A'_1|\cdot |A'_2-A'_3|}{|x+A'_2|\cdot |x+A'_3|} \\\le 
\Bigl(2\frac{1/8+\gamma_1}{(x-1/2-\gamma_1)^2}+\frac{1+2\gamma_1}{(x-1/2-\gamma_1)^2}\Bigr)\|A-A'\|\\ 
\le \frac{8\|A-A'\|}{x^2+1},\qquad |x|\ge 2.
\end{multline*}

Since
$$
\mathbb E\sum_{m\in\mathbb Z}\frac{|\zeta_m|}{m^2+1}<\infty, 
$$
almost surely we have 
$$
|\zeta_m|\le C_\zeta+ m^2,\qquad m\in\mathbb Z, 
$$
for some (random) $C_\zeta$ depending on $\zeta$. 

Let $T\ge 100$ be a large integer number to be fixed later on and let $K\ge 2C_\zeta$ be a large number to be fixed later on. 
Set 
\begin{multline*}
S=\Bigl\{n\in\mathbb Z: \Bigl(\frac{\zeta_{Tn-1}\omega(Tn-1)}{\omega(Tn)},\zeta_{Tn},\frac{\zeta_{Tn+1}\omega(Tn+1)}{\omega(Tn)}\Bigr)\\
\in D(LA^*,\gamma_2/4)\Bigr\}.
\end{multline*}
Next, set 
\begin{align*}
Y_k&=\Bigl\{n\in[2^k,2^{k+1}-1]:\frac{\omega(n)}{\omega(n+1)}<\frac{k}{k+1}\Bigr\},\qquad k\ge 0,\\
Y&=\cup_{k\ge 0}Y_k.
\end{align*}
Since $\omega$ is decreasing and $\omega(2x)\ge \delta \omega(x)$, we have
$$
\card(Y_k)=O(k),\qquad k\to\infty.
$$

By the Kolmogorov strong law of large numbers (see, for instance, \cite[Section 4.3.2]{shi}), for some absolute constant 
$\varepsilon>0$ we have almost surely
$$
\lim_{R\to\infty}\frac{\card(S\cap(-R,R))}{2R}= \varepsilon.
$$
Furthermore, by the Kolmogorov law of the iterated logarithm (see, for instance, \cite[Theorem 7.1]{pet}), 
applied to the characteristic functions of the events ($n\in S$, $Tn-1,Tn\not\in Y$) we have almost surely
\begin{equation}
\limsup_{R\to\infty}\frac{|\card (S\cap [0,R))-\varepsilon R|+|\card (S\cap(-R,0))-\varepsilon R|}{\sqrt{R\log\log R}}<\infty.
\label{lalim1}
\end{equation}
Next, we set 
\begin{gather*}
\Delta_n=\{Tn-1,Tn,Tn+1\},\\
U=\mathbb Z\setminus\bigcup_{n\in S}\Delta_n.
\end{gather*}
For $T\ge 3$, the sets $\Delta_n$ are disjoint.

To find a meromorphic $F$ (an infinite linear combination of the Cauchy kernels with poles on $\mathbb R$) such that 
$F(m)=\zeta_m\omega(m)$, $m\in\mathbb Z$, we are going to solve a system of (non-linear) equations 
\begin{multline}
\zeta_m\omega(m)=\sum_{s\in U}\frac{\alpha_s\omega(s)/(K(s^2+1))}{m-(s-K^{-2}(s^2+1)^{-2})}\\+
\sum_{n\in S}\omega(Tn)f_{(\alpha_{Tn-1},\alpha_{Tn},\alpha_{Tn+1})}(m-Tn),\qquad m\in\mathbb Z,
\label{delta4}
\end{multline}
for some $\alpha=(\alpha_s)\in\ell^\infty$. 

We obtain $\alpha$ using a fixed-point type iterative argument. To deal with the sum in the right hand side of \eqref{delta4}, we divide it into the local part $W$ and the non-local part $V$ and study these parts separately.

First, we define
$$
\eta_m=\begin{cases}
\dfrac{\zeta_m}{K(m^2+1)},\qquad m\in U,\\
\dfrac{\zeta_m\omega(m)}{\omega(Tn)},\qquad m\in \Delta_n,\, n\in S.
\end{cases}
$$
Then almost surely $\eta=(\eta_m)_{m\in\mathbb Z}\in\ell^\infty$, $\|\eta\|\le 3/4$. Here and later on, $\|x\|=\|x\|_\infty$, $x\in\ell^\infty$. 

Now we introduce 
$$
W_m:a\in\ell^\infty \mapsto 
\begin{cases}
a_m,\qquad m\in U,\\
(L(a_{Tn-1},a_{Tn},a_{Tn+1}))_{m-Tn+2}, \quad m\in \Delta_n,\, n\in S.
\end{cases}
$$
Given $a\in\ell^\infty$ and $m\in U$, we define
\begin{multline*}
V_ma=\frac1{K\omega(m)(m^2+1)}\Bigl[
\sum_{s\in U\setminus\{m\}}\frac{a_s\omega(s)/(K(s^2+1))}{m-(s-K^{-2}(s^2+1)^{-2})}\\
+\sum_{n\in S}\omega(Tn)f_{(a_{Tn-1},a_{Tn},a_{Tn+1})}(m-Tn)
\Bigr],
\end{multline*}
the second series converges because $f_A(x)$ decay quadratically as $|x|\to\infty$. 

Given $a\in\ell^\infty$ and $m\in \Delta_n$, $n\in S$, we define
\begin{multline*}
V_ma=\frac1{\omega(Tn)}\Bigl[
\sum_{s\in U}\frac{a_s\omega(s)/(K(s^2+1))}{m-(s-K^{-2}(s^2+1)^{-2})}\\
+\sum_{y\in S\setminus\{n\}}\omega(Ty)f_{(a_{Ty-1},a_{Ty},a_{Ty+1})}(m-Ty)
\Bigr].
\end{multline*}
Finally, we set
\begin{align*}
W:a\in\ell^\infty &\mapsto (S_ma)_{m\in \mathbb Z},\\
V:a\in\ell^\infty &\mapsto (V_ma)_{m\in \mathbb Z}.
\end{align*}
Given $\gamma>0$, we define
\begin{align*}
E_{1,\gamma}&=\{a\in\ell^\infty:\|a\|\le 1,\,(a_{Tn-1},a_{Tn},a_{Tn+1})\in D(A^*,\gamma), \,n\in S \},\\
E_{2,\gamma}&=\{a\in\ell^\infty:\|a\|\le 1,\,(a_{Tn-1},a_{Tn},a_{Tn+1})\in D(LA^*,\gamma), \,n\in S \}.
\end{align*}
By the definition of $S$, we have $\eta\in E_{2,\gamma_2/4}$.

If $a\in E_{2,\gamma_2}$, then, since $L|_{D(A^*,\gamma_1)}$ is a bijection, and by \eqref{lat}, we can define 
\begin{equation}
W^{-1}a=b\in E_{1,\gamma_1}\text{\ \ such that \ }Wb=a.\label{laa}
\end{equation}
If $m\in U$, then $(W^{-1}x)_m=x_m$. 

\begin{lemma} For 
every $\tau>0$, if $T\ge T(\tau,\delta)$, $K\ge \max(2C_\zeta,K(\tau,\delta))$, then 
\begin{align}
\|W^{-1}x-W^{-1}y\|&\le 3\|x-y\|, \qquad x,y\in E_{2,\gamma_2},\label{last}\\
\|VW^{-1}\eta\|&\le \tau, \label{last2}\\
\|Vx-Vy\|&\le \tau\|x-y\|, \qquad x,y\in E_{1,\gamma_1}.\label{last4} 
\end{align}
\label{ler5}
\end{lemma}

\begin{proof}[Proof of Lemma]
To prove \eqref{last} we use that if $m\in U$, then \newline\noindent $(W^{-1}x-W^{-1}y)_m=x_m-y_m$. Otherwise, if $m\in \Delta_n$, $n\in S$, then we use \eqref{latt}.

Next we pass to the proof of \eqref{last2}. By \eqref{laa}, $\|W^{-1}\eta\|\le 1$ and 
$((W^{-1}\eta)_{Tn-1},(W^{-1}\eta)_{Tn},(W^{-1}\eta)_{Tn+1})\in D(A^*,\gamma_1)$, $n\in\mathbb Z$. If $m\in U$, then by \eqref{lain}
\begin{multline*}
|(VW^{-1}\eta)_m|=|V_m W^{-1}\eta|
\\ \le \frac1{K\omega(m)(m^2+1)}\Bigl[
\sum_{s\in U\setminus\{m\}}\frac{|(W^{-1}\eta)_s|\omega(s)/(K(s^2+1))}{m-(s-K^{-2}(s^2+1)^{-2})}\\
+\sum_{n\in S}\omega(Tn)|f_{((W^{-1}\eta)_{Tn-1},(W^{-1}\eta)_{Tn},(W^{-1}\eta)_{Tn+1})}(m-Tn)|
\Bigr],
\\
\le \frac{2}{K\omega(m)(m^2+1)} \Bigl[
\sum_{s\in U\setminus\{m\}}\frac{\omega(s)}{K(s^2+1)|m-s|}
+\sum_{n\in S}\frac{\omega(Tn)}{(m-Tn)^2+1}
\Bigr].
\end{multline*}
Let us first verify that 
$$
\sum_{s\in U\setminus\{m\}}\frac{\omega(s)}{K^2\omega(m)(m^2+1)(s^2+1)|m-s|}\le \frac{\tau}{4}
$$
for $K\ge K(\tau,\delta)$, $T\ge 3$. Without loss of generality, we can assume that $m\ge 0$.

Now,
\begin{multline*}
\sum_{s\in U\setminus\{m\}}\frac{\omega(s)}{ \omega(m)(m^2+1)(s^2+1)|m-s|}\\ \lesssim
\sum_{0\le s\le m/2}\frac{\omega(s)}{ \omega(m)(m^2+1)(s^2+1)|m-s|}\\+\sum_{m/2< s< m}\frac{\omega(s)}{ \omega(m)(m^2+1)(s^2+1)|m-s|}\\+\sum_{s>m}\frac{\omega(s)}{ \omega(m)(m^2+1)(s^2+1)|m-s|}\\ \lesssim
\sum_{0\le s\le m/2}\frac{1}{ \delta (m^2+1)(s^2+1)}+\sum_{m/2< s< m}\frac{1}{ \delta(m^2+1)^2}\\+\sum_{s>m}\frac{1}{ (m^2+1)(s^2+1)}\le c(\delta).
\end{multline*}

Next, we verify that 
$$
\sum_{n\in S}\frac{\omega(Tn)}{K\omega(m)(m^2+1)((m-Tn)^2+1)}\le \frac{\tau}{4}
$$
for $K\ge K(\tau,\delta)$, $T\ge 3$, $m\ge 0$.

Indeed,
\begin{multline*}
\sum_{n\in S}\frac{\omega(Tn)}{\omega(m)(m^2+1)(m-Tn)^2}\\ \lesssim
\sum_{0\le n\le m/(2T)}\frac{\omega(Tn)}{\omega(m)(m^2+1)(m-Tn)^2}\\
+\sum_{m/(2T)< s< m/T}\frac{\omega(Tn)}{\omega(m)(m^2+1)(m-Tn)^2}\\
+\sum_{n>m/T}\frac{\omega(Tn)}{\omega(m)(m^2+1)(m-Tn)^2}\\ \lesssim
\sum_{0\le n\le m/(2T)}\frac{1}{ \delta (m^2+1)}+\sum_{m/(2T)< s< m/T}\frac{1}{ \delta(m^2+1)}+\sum_{n>m/T}\frac{1}{ m^2+1}\\ \le c(\delta).
\end{multline*}

Summing up, we obtain that 
$|(VW^{-1}\eta)_m|\le \tau$ for $K\ge K(\tau,\delta)$, $T\ge 3$.

If now $m\in \Delta_n$, $n\in S$, then again by \eqref{lain} we have 
\begin{multline*}
|(VW^{-1}\eta)_m|=|V_mW^{-1}\eta|\le 
\frac1{\omega(Tn)}\Bigl[
\sum_{s\in U}\frac{|(W^{-1}\eta)_s|\omega(s)/(K(s^2+1))}{m-(s-K^{-2}(s^2+1)^{-2})}\\
+\sum_{y\in S\setminus\{n\}}\omega(Ty)|f_{((W^{-1}\eta)_{Ty-1},(W^{-1}\eta)_{Ty},(W^{-1}\eta)_{Ty+1})}(m-Ty)|
\Bigr]\\ \le
\frac{2}{\omega(Tn)}\Bigl[
\sum_{s\in U}\frac{\omega(s)}{K(s^2+1)|m-s|}
+\sum_{y\in S\setminus\{n\}}\frac{\omega(Ty)}{(m-Ty)^2+1} 
\Bigr].
\end{multline*}
Let us first verify that 
$$
\sum_{s\in U}\frac{\omega(s)}{K\omega(Tn)(s^2+1)|m-s|}\le \frac{\tau}{4} 
$$
for $K\ge K(\tau,\delta)$, $T\ge 3$, $m\ge 0$.

Indeed,
\begin{multline*}
\sum_{s\in U\setminus\{m\}}\frac{\omega(s)}{ \omega(Tn)(s^2+1)|m-s|} \lesssim
\sum_{0\le s\le m/2}\frac{\omega(s)}{ \omega(Tn)(s^2+1)|m-s|}\\+
\sum_{m/2< s< m}\frac{\omega(s)}{ \omega(Tn)(s^2+1)|m-s|}+
\sum_{s>m}\frac{\omega(s)}{ \omega(Tn)(s^2+1)|m-s|}\\ \lesssim
\sum_{0\le s\le m/2}\frac{1}{ \delta (s^2+1)}+\sum_{m/2< s< m}\frac{1}{\delta  (m^2+1)}+\sum_{s>m}\frac{1}{ m^2+1}\le c(\delta).
\end{multline*}

Next, we verify that 
$$
\sum_{y\in S\setminus\{n\}}\frac{\omega(Ty)}{\omega(Tn)((m-Ty)^2+1)} \le \frac{\tau}{4}
$$
for $T\ge T(\tau,\delta)$, $m\ge 0$.

Indeed,
\begin{multline*}
\sum_{y\in S\setminus\{n\}}\frac{\omega(Ty)}{\omega(Tn)T^2(n-y)^2} \lesssim
\sum_{0\le y\le n/2}\frac{\omega(Ty)}{\omega(Tn)T^2(n-y)^2}\\
+\sum_{n/2< y< n}\frac{\omega(Ty)}{\omega(Tn)T^2(n-y)^2}
+\sum_{y>n}\frac{\omega(Ty)}{\omega(Tn)T^2(n-y)^2}\\ \lesssim
\sum_{0\le y\le n/2}\frac{1}{ \delta Tn}+\sum_{n/2< y< n}\frac{1}{ \delta T^2(n-y)^2}+\sum_{n>m/T}\frac{1}{ T^2(n-y)^2}\le \frac{c(\delta)}T.
\end{multline*}
Summing up, we obtain that 
$|(VW^{-1}\eta)_m|\le \tau$ for $K\ge K(\tau,\delta)$, $T\ge T(\tau,\delta)$. This establishes \eqref{last2}.  

It remains to verify \eqref{last4}. 
If $m\in U$, then by \eqref{lar} we have 
\begin{multline*}
|(Vx)_m-(Vy)_m|=|V_mx-V_my|\le 
\frac{10\|x-y\|}{K\omega(m)(m^2+1)}\\\times \Bigl[
\sum_{s\in U\setminus\{m\}}\frac{\omega(s)/(K(s^2+1))}{|m-s|}+\sum_{n\in S}\frac{\omega(Tn)}{(m-Tn)^2+1}
\Bigr]\le\tau\|x-y\| 
\end{multline*}
when  $K\ge K(\tau,\delta)$, $T\ge 1$.

Finally, if $m\in \Delta_n$, $n\in S$, then again by \eqref{lar} we have 
\begin{multline*}
|(Vx)_m-(Vy)_m|=|V_mx-V_my|\\ \le 
\frac{10\|x-y\|}{\omega(Tn)}\Bigl[
\sum_{s\in U}\frac{\omega(s)/(K(s^2+1))}{|m-s|}
+\sum_{y\in S\setminus\{n\}}\frac{\omega(Ty)}{(m-Ty)^2+1} 
\Bigr]\\ \le\tau\|x-y\|
\end{multline*}
when  $K\ge K(\tau,\delta)$, $T\ge T(\tau,\delta)$.
This establishes \eqref{last4}.  
\end{proof}

Let us return to the proof of the proposition. 
Given $\tau\in(0,1/4)$, from now on we assume that $K\ge \max(2C_\zeta,K(\tau,\delta))$, $T\ge T(\tau,\delta)$ so that the assertions of Lemma~\ref{ler5} are satisfied. 
 
We set 
\begin{equation}
\label{stab}\begin{cases}
\alpha(0)= W^{-1}\eta,\\
\alpha(j+1)= W^{-1}(\eta-V\alpha(j)),\qquad j\ge 0.
\end{cases}
\end{equation}

We have $\eta\in E_{2,\gamma_2/4}$, $\|\eta\|\le 3/4$. By \eqref{laa}, $\alpha(0)\in E_{1,\gamma_1}$. Next, $\|\alpha(0)\|\le 3/4$.   
By \eqref{last2}, $\|V\alpha(0)\|\le \tau$. 
If $\tau<\gamma_2/4$, then 
\begin{equation}
\eta-V\alpha(0)\in E_{2,\gamma_2/2}. \label{ladop}
\end{equation}
By \eqref{last},
$$
\|\alpha(1)-\alpha(0)\|\le 2\|V\alpha(0)\|\le 2\tau. 
$$
If $\tau<\gamma_1/4$, then 
$$
\|\alpha(1)-\alpha(0)\|\le \gamma_1/2.
$$
Hence, $\alpha(1)\in E_{1,\gamma_1}$.

Next, by \eqref{last4}, 
\begin{multline*}
\|V\alpha(1)\|\le \|V\alpha(1)-V\alpha(0)\|+\|V\alpha(0)\|\\ \le \tau \|\alpha(1)-\alpha(0)\|+\|V\alpha(0)\|\le 2\tau.
\end{multline*}
If $\tau<\gamma_2/4$, then $\eta-V\alpha(1)\in E_{2,\gamma_2}$. 

Now, for $j\ge 1$, suppose that 
$$
\|\alpha(s)-\alpha(s-1)\|\le \gamma_12^{-s},\qquad 1\le s\le j,
$$
and $\eta-V\alpha(j-1), \eta-V\alpha(j)\in E_{2,\gamma_2}$, $\alpha(j-1),\alpha(j)\in E_{1,\gamma_1}$. Then, by \eqref{laa},  
$\alpha(j+1)=W^{-1}(\eta-V\alpha(j))\in E_{1,\gamma_1}$.
By \eqref{last} and \eqref{last4}, 
\begin{multline*}
\|\alpha(j+1)-\alpha(j)\|=\|W^{-1}(\eta-V\alpha(j))-W^{-1}(\eta-V\alpha(j-1))  \| \\
\le 2\|V\alpha(j)-V\alpha(j-1)\|\le 2\tau \|\alpha(j)-\alpha(j-1)\| \\ 
\le2\tau \gamma_12^{-j}\le \gamma_12^{-(j+1)}.
\end{multline*}
Again by \eqref{last4}, we have 
$$
\|V\alpha(j+1)-V\alpha(0))\|\le  \tau \|\alpha(j+1)-\alpha(0)\|\le \tau\gamma_1. 
$$
If $\tau\gamma_1<\gamma_2/2$, then, by \eqref{ladop}, $\eta-V\alpha(j+1)\in E_{2,\gamma_2}$.
This completes the induction step under the condition that
$$
\tau<\min(\gamma_1/4,\gamma_2/4).
$$
Fix such $\tau$ and fix some $K\ge \max(2C_\zeta,K(\tau,\delta))$, $T\ge T(\tau,\delta)$. Then our process \eqref{stab} gives a sequence $(\alpha(j))_{j\ge 0}$ 
that converges in $\ell^\infty$ to a point $\alpha\in E_{1,\gamma_1}$.
By \eqref{last} and \eqref{last4} we obtain that $\alpha=W^{-1}(\eta-V\alpha)$, and hence, 
$$
W\alpha+V\alpha=\eta.
$$ 

We set
\begin{multline*}
F(z)=\sum_{s\in U}\frac{\alpha_s\omega(s)/(K(s^2+1))}{z-(s-K^{-2}(s^2+1)^{-2})}\\
+\sum_{y\in S}\omega(Ty)f_{(\alpha_{Ty-1},\alpha_{Ty},\alpha_{Ty+1})}(z-Ty).
\end{multline*}
The two series converge and determine the meromorphic function $F$ with simples poles at the points $s-K^{-2}(s^2+1)^{-2}$, $s\in U$, and at the 
points $Ty-\alpha_{Ty}$, $Ty-\alpha_{Ty+1}$, $y\in S$. 
Therefore, by \eqref{lalim1}, condition \eqref{lalim} is satisfied with $\varepsilon=\varepsilon(\delta)>0$. 
Property \eqref{d03} follows by construction. 

If $m\in U$, then 
\begin{multline*}
F(m)=K(m^2+1)\omega(m)\alpha_m+\sum_{s\in U\setminus \{m\}}\frac{\alpha_s\omega(s)/(K(s^2+1))}{m-(s-K^{-2}(s^2+1)^{-2})}\\
+\sum_{y\in S}\omega(Ty)f_{(\alpha_{Ty-1},\alpha_{Ty},\alpha_{Ty+1})}(m-Ty)\\=
K(m^2+1)\omega(m)((W\alpha)_m+(V\alpha)_m)=K(m^2+1)\omega(m)\eta_m=\zeta_m\omega(m).
\end{multline*}
If now $m\in \Delta_n$, $n\in S$, then 
\begin{multline*}
F(m)=\omega(Tn)f_{(\alpha_{Tn-1},\alpha_{Tn},\alpha_{Tn+1})}(m-Tn)\\+
\sum_{s\in U}\frac{\alpha_s\omega(s)/(K(s^2+1))}{m-(s-K^{-2}(s^2+1)^{-2})}
+\sum_{y\in S\setminus\{n\}}\omega(Ty)f_{(\alpha_{Ty-1},\alpha_{Ty},\alpha_{Ty+1})}(m-Ty)\\=\omega(Tn)L((\alpha_{Tn-1},\alpha_{Tn},\alpha_{Tn+1}))_{m-Tn+2}+
\omega(Tn)V_m\alpha\\=
\omega(Tn)((W\alpha)_m+(V\alpha)_m)=\omega(Tn)\eta_m=\zeta_m\omega(m).
\end{multline*}
Thus, $F(m)=\zeta_m\omega(m)$, $m\in\mathbb Z$, and the proof is completed.
\end{proof}

The same result holds if $(\zeta_m)$, $m\in\mathbb Z$, are independent standard Gaussian complex variables. 

\begin{proposition} There exists $\varepsilon>0$ 
such that if $(\zeta_m)_{m\in\mathbb Z}$ is the sequence of independent standard Gaussian complex variables, then, almost surely,
there exists a meromorphic function $F$, whose set of poles $Q=(q_k)_{k\in\mathbb Z}$ is a separated subset of the real line, such that 
$$
F(m)=\zeta_m\omega(|m|),\qquad m\in\mathbb Z,
$$
$$
|F(z)|\lesssim 1+\frac1{\dist(z,Q)}, \qquad z\in\mathbb C,
$$
and
$$
|n_Q(t)-(1-\varepsilon)t|\lesssim 1+|t|^{2/3},\qquad t\in\mathbb R.
$$
\label{thm1c}
\end{proposition}

The only change in the proof with respect to that of Proposition~\ref{thm1} is that we replace the linear combination $f_A$ of two Cauchy kernels by that of four Cauchy kernels.


\section{Approximation results}
\label{sec3}

We use here some standard notation. Given an entire function $f$ of exponential type, we denote by $\mathcal Z(f)$ its zero set and by $\Type(f)$ its exponential type. 

Denote the Cartwright class by $\Cart$ and the Paley--Wiener space by $\Pw$. 
Let ${\bf k}_\lambda$ be the reproducing kernel of $\Pw$ at $\lambda\in\mathbb C$. 
The standard Fourier transform $\mathcal F$ maps $L^2[-\pi,\pi]$ onto $\Pw$ and the exponentials transform into the reproducing kernels. Therefore, \eqref{mainincl} is equivalent to 
$$
\Pw\ni \mathcal F f\in\Spa_{\lambda\in\Lambda}{\bf k}_\lambda.
$$

The following two results are probably known to experts. For the sake of completeness, we give their proofs here. 

\begin{lemma}
Let $F$ be a meromorphic function, whose set of poles $Q=(q_k)_{k\in\mathbb Z}$ is a separated subset of the real line, such that
\begin{equation}
|F(z)|\lesssim 1+\frac1{\dist(z,Q)}, \qquad z\in\mathbb C,
\label{d06}
\end{equation}
and let for some $C>0$ we have 
\begin{equation}
|n_Q(t)-Ct|\lesssim 1+|t|^{2/3},\qquad t\in\mathbb R,
\label{d07}
\end{equation}
where $n_Q$ is given by \eqref{d02}.

If $V$ is the canonical product constructed by $Q$, then $F=U/V$, 
where $U,V\in\Cart$ and $\Type (U)\le \Type (V)=C\pi$.
\label{lem22}
\end{lemma}

\begin{proof}
Without loss of generality, $0\not\in Q$. Set $\rho(t)=n_Q(t)-Ct$. 
Since $n_Q(t)=O(|t|)$, $|t|\to\infty$, the infinite product $\prod_{q\in Q}(1-z/q)\exp(-z/q)$ converges in the whole complex plane. 
By condition \eqref{d07}, we can define
$$
V(z)=\lim_{R\to\infty}\prod_{q\in Q,\,|q|<R}\Bigl(1-\frac{z}{q}\Bigr),
$$ 
with the products converging uniformly on compacts. Furthermore, we have 
\begin{multline*}
\log|V(z)|=\Re \lim_{R\to\infty}\int_{-R}^{R}\log\Bigl(1-\frac{z}t\Bigr)\,dn_Q(t)\\=C\Re \int_0^\infty\log\Bigl(1-\frac{z^2}{t^2}\Bigr)\,dt+\Re \int_{-\infty}^{\infty}\log\Bigl(1-\frac{z}t\Bigr)\,d\rho(t)\\=
C\pi|\Im z|+\Re \int_{-\infty}^{\infty}\log\Bigl(1-\frac{z}t\Bigr)\,d\rho(t).
\end{multline*}
Next, 
$$
\Re \int_{-\infty}^{\infty}\log\Bigl(1-\frac{z}t\Bigr)\,d\rho(t)=\Re \int_{-\infty}^{\infty}\frac{z\rho(t)}{t(z-t)}\,dt,
$$
and 
$$
 \int_{-\infty}^{\infty}\frac{|z|}{|z-t|}\cdot \frac{|\rho(t)|}{|t|}\,dt=\int_{|t|<|z|/2}+\int_{|t|>2|z|}+\int_{|z|/2\le |t|\le 2|z|}=I_1+I_2+I_3.
$$
By \eqref{d07} we have 
$$
I_1+I_2\lesssim 1+|z|^{2/3}.
$$
Furthermore,
$$
I_3\lesssim (1+|z|^{2/3})\int_{|z|/2\le |t|\le 2|z|}\frac{dt}{|z-t|}\lesssim \frac{1+|z|^{4/5}}{|\Im z|}.
$$
Thus, for some $M<\infty$, 
$$
\log|V(z)|\ge -M\frac{1+|z|^{4/5}}{|\Im z|}, 
$$
and by Matsaev's theorem \cite[Section 26.4]{L}, we have $V\in\Cart$.

Since $F$ has only simple zeros at the points of $Q$, and $V$ vanishes there, $U=VF$ is an entire function. 
By \eqref{d06}, $U$ is of at most exponential growth outside small disks around points in $Q$, and by the maximum principle, 
$U$ is of exponential type. Since $V\in\Cart$, $\int_{\mathbb R}\frac{\log^+|V(x+i)|}{1+x^2}\,dx<\infty$ and, again by \eqref{d06}, 
we have $\int_{\mathbb R}\frac{\log^+|U(x+i)|}{1+x^2}\,dx<\infty$.
Hence, $U\in\Cart$.
\end{proof}

A disjoint sequence of intervals $(I_n)$ on the real line is said to be a long sequence of intervals if $|I_n|\to\infty$ and 
\begin{equation}
\sum_n \frac{|I_n|^2}{1 + {\rm dist}^2(0,I_n)} = \infty.
\label{d147}
\end{equation}
Given $\Lambda\subset \mathbb R$, by definition, its Beurling--Malliavin density $D_{BM}(\Lambda)$ is the supremum of $d$ such that there exists a long sequence of intervals $(I_n)$ 
satisfying the relation $\card(\Lambda\cap I_n)>d|I_n|$.

\begin{lemma} In the conditions of Lemma~\ref{lem22}, we have 
$$
D(Q)=D_{BM}(Q)=C.
$$
\label{lem28}
\end{lemma}

\begin{proof} If $I_n=(2^n,2^{n+1})$, $n\ge 1$, then $(I_n)$ is a long sequence of intervals and by \eqref{d07}, 
$\card(Q\cap I_n)\ge C|I_n|+O(|I_n|^{3/2})$, hence, $D_{BM}(Q)\ge C$.

Next, if $d>C$ and if $(I_n)=(a_n,b_n)$ is a long sequence of intervals such that $\card(Q\cap I_n)>d|I_n|$, then 
$$
d(b_n-a_n)<n_Q(b_n)-n_Q(a_n)<Cb_n-Ca_n+O(b_n^{3/2}),
$$
and
$$
|I_n|=O(|I_n|^{2/3}+{\rm dist}^2(0,I_n)^{2/3})
$$
which contradicts to \eqref{d147}. 

The equality $D(Q)=C$ is immediate. 
\end{proof}

Next we show that approximation by a fixed family of reproducing kernels  is equivalent to interpolation by meromorphic functions on $\mathbb Z$. 

\begin{theorem}
Let $F\in\Pw$, let $V\in\Cart$ be a real entire function with real simple zeros and with $\Type (V)<\pi$, and let $\mathcal Z(V)\cap\mathbb Z=\emptyset$.  
Then 
\begin{equation}
F\in\Spa_{\lambda\in\mathcal Z(V)}{\bf k}_\lambda
\label{d09}
\end{equation}
if and only if there exists $U\in\Cart$ such that $\Type(U)\le \Type(V)$ and 
\begin{equation}
F(n)=(-1)^n\frac{U(n)}{V(n)},\qquad n\in\mathbb Z.
\label{d10}
\end{equation}
\label{thm22}
\end{theorem}

\begin{proof} Suppose that $F$ satisfies \eqref{d10}. 

Here we use the Beurling--Malliavin multiplier theorem \cite[Section X.A]{K} in the following form:

{\it for every $f\in\Cart$ and for every $\varepsilon>0$ there exists an entire function $\varphi\not=0$ of exponential type at most $\varepsilon$ such that $\varphi f\in L^{\infty}(\mathbb R)$.} 

Let us recall that there is some freedom in the choice of $\varphi$. In particular, if we shift its zeros, say, by exponentially small distances, then the assertion will still hold.

Given an entire function $f$, we define $f^*(z)=\overline{f(\overline{z}})$.

We apply the Beurling--Malliavin multiplier theorem to $f=UU^*+V^2+1$ with $\varepsilon<\pi-\Type (V)$, 
obtain $\varphi$ and multiply it by $\sin (\varkappa z)/z$ with $\varkappa=\pi-\Type (V)-\varepsilon$ 
to get a real entire function $W\in\Cart$ such that $\mathcal Z(W)\cap \mathcal Z(V)=\emptyset$ and 
$$
UW,VW\in\Pw,\quad \Type(VW)=\pi.
$$

Now, we are going to use (a version of) a powerful result from \cite{ABB}. 
The proof there is sufficiently technically involved and includes several elements used in similar situations in \cite{BBB2012}, \cite[Section 4.2]{BB2023}. Here we need a slight modification of this result and just indicate necessary changes in the argument. 
Given $a<b$ we define 
${\tt PW_{[a,b]}}=\mathcal F(L^2[a,b])$, where $\mathcal F$ is the Fourier transform, $\mathcal Ff(z)=\int f(t)e^{itz}\,dt$. 
The following assertion is proved in \cite[Proposition 2.1]{ABB}:

{\it Let $G\in\Pw$ be such that
$$
\text{for every proper subinterval $[a,b]$ of $[-\pi,\pi]$,\ } G\not\in {\tt PW_{[a,b]}}.
$$
Suppose that $\mathcal Z(G)=\Lambda_1 \sqcup \Lambda_2$, $\Lambda_2\subset\mathbb R$, and suppose that 
\begin{equation}
\mathcal D_+(\Lambda_2)=\lim_{R\to\infty}\sup_{x\in\mathbb R} \frac{\card(\Lambda_2\cap[x,x+R])}{R}<1.
\label{fr12}
\end{equation}
Then the mixed system
$$
\bigl({\bf k}_\lambda\bigr)_{\lambda\in\Lambda_2}\cup \Bigl(\frac{VW}{\cdot-\mu}\Bigr)_{\mu\in \Lambda_1} 
$$
is complete in $\Pw$. }

We are going to use a variant of this assertion with \eqref{fr12} replaced by the condition that $\Lambda_2=\mathcal Z(V)$ for some 
real entire function $V\in\Cart$ with simple real zeros such that $\Type (V)<\pi$. As in the proof of \cite[Proposition 2.1]{ABB} 
we arrive at the formula (3.7) there:
\begin{equation}
n_{\mathcal Z(V)}(x)+n_\Sigma(x)=   x +\widetilde {u}(x)+v(x) +\alpha(x),\qquad x\in\mathbb R,
\label{BM}
\end{equation}
where $\Sigma\subset\mathbb R$ is a union of two sequences, each of them separated, $u\in L^1(dx/(1+x^2))$, $\widetilde {u}$ is the Hilbert transform of $u$,
$$
\widetilde {u}(x)=\frac1\pi v.p.\int_\mathbb R \Bigl( \frac1{x-t}+\frac{t}{t^2+1}u(t)\,dt \Bigr),
$$
$v\in L^\infty(\mathbb R)$, $\alpha$ is a nondecreasing function.  Furthermore, we get a non-polynomial entire function $R$ of zero exponential type which is bounded on $\Sigma$. 

A sequence $\Sigma\subset\mathbb R$ is said to be a P\'olya sequence if any entire function of zero exponential type which is
bounded on $\Sigma$ is a constant. 
As in the proof of \cite[Proposition 2.1]{ABB} we use a characterization of P\'olya sequences obtained by Mitkovski--Poltoratski \cite[Theorem X]{MP}. Namely, the fact that $R$ is bounded on $\Sigma$ implies that 
there exists a long system of intervals 
$(I_n)$ such that $\card(\Sigma\cap I_n)/|I_n|\to 0$.  

Since $\Type (V)<\pi$, by the second Beurling--Malliavin theorem we can find $\varepsilon>0$ and a long system of intervals 
$(J_n)$ such that 
\begin{equation} 
\frac{\card(\Sigma\cap J_n)}{|J_n|}\le \varepsilon, \quad\frac{\card(\mathcal Z(V)\cap J_n)}{|J_n|}\le 1-5\varepsilon.
\label{BM1}
\end{equation}
Given an interval $I=[a,b]\subset\mathbb R$ and a function $\gamma$, we set 
$$
\Delta_I[\gamma]=\inf_{[\varepsilon a+(1-\varepsilon)b,b]}\gamma-\inf_{[a,(1-\varepsilon)a+\varepsilon b)]}\gamma.
$$
Set $\gamma(x)=\pi x +v(x) -n_{\mathcal Z(V)}(x)-n_\Sigma(x)$. 
Then, by \eqref{BM1}, we have
$$
\Delta_{J_n}[\gamma]\ge \varepsilon |J_n|,\qquad n\ge n_0.
$$
A version of the second Beurling--Malliavin theorem (see \cite[Proposition 3.1]{ABB} and the references to the work of 
Makarov--Poltoratski \cite{MAP} there) tells that such $\gamma$ 
cannot be represented as $\widetilde {w}-\alpha$, where $\alpha$ is nondecreasing and $w\in L^1(dx/(1+x^2))$.
By \eqref{BM}, $\gamma=-\widetilde {u}-\alpha$, and we get a contradiction which completes the proof.

Now we apply our assertion to $G=VW$, $\Lambda_1=\mathcal Z(W)$, $\Lambda_2=\mathcal Z(V)$. 
Then the mixed system
$$
\bigl({\bf k}_\lambda\bigr)_{\lambda\in\mathcal Z(V)}\cup \Bigl(\frac{VW}{\cdot-\mu}\Bigr)_{\mu\in \mathcal Z(W)} 
$$
is complete in $\Pw$.  

Therefore, to prove \eqref{d09}, it suffices to verify that 
$$
F\perp \frac{VW}{\cdot-\mu},\qquad  \mu\in \mathcal Z(W),
$$
that is 
$$
\sum_{k\in\mathbb Z}\frac{F(k)V(k)W(k)}{k-\mu}=0,\qquad \mu\in \mathcal Z(W),
$$
or, by \eqref{d10},
$$
\sum_{k\in\mathbb Z}(-1)^k\frac{U(k)W(k)}{k-\mu}=0,\qquad \mu\in \mathcal Z(W).
$$
The latter relation follows from the equality 
$$
\sum_{k\in\mathbb Z}(-1)^k\frac{U(k)W(k)}{k-z}=\frac{\pi U(z)W(z)}{\sin \pi z},
$$
which is just the Kotelnikov--Nyquist--Shannon--Whittaker formula.

In the opposite direction, suppose that $F$ satisfies \eqref{d09}. Then we have   
\begin{equation}
H|\Lambda=0 \implies H\perp F,\qquad H\in \Pw. 
\label{d11}
\end{equation}

We introduce the space 
$$
\mathcal H_V=\bigl\{f\in {\rm Hol}(\mathbb C):fV\in\Pw\bigr\}
$$
with the norm given by 
$$
\|f\|_{\mathcal H_V}=\|fV\|_{\Pw}. 
$$
For every $w\in\mathbb C\setminus\mathbb R$, the evaluation functional $\mathcal H_V\ni F \mapsto F(w)$ is continuous. 
If $(f_n)$ is a Cauchy sequence in $\mathcal H_V$, then $(f_nV)$ is a Cauchy sequence in $\Pw$, and, hence, converges to 
a function $g\in\Pw$. Next, $\mathcal Z(g)\supset \mathcal Z(V)$, and, hence, $f=g/V\in {\rm Hol}(\mathbb C)$. Therefore, 
$(f_n)$ converges to $f$ in $\mathcal H_V$.

Therefore, $\mathcal H_V$ is a Hilbert space of entire functions. Furthermore, if $f\in\mathcal H_V$ and if $f(\lambda)=0$,  
$f_\lambda(z)=\frac{z-\overline{\lambda}}{z-\lambda}f(z)$, then $f^*,f_\lambda\in\mathcal H_V$, and 
$$
\|f^*\|_{\mathcal H_V}=\|f_\lambda\|_{\mathcal H_V}=\|f\|_{\mathcal H_V}.
$$ 
Thus, $\mathcal H_V$ is a de Branges space of entire functions, see for example, \cite{B,deB}. 
Choose $\varepsilon\in (0,(\pi-\Type(V))/2)$. By the Beurling--Malliavin multiplier theorem, we find an entire function $\varphi\not=0$ 
of exponential type at most $\varepsilon$ such that $\varphi V\in L^\infty(\mathbb R)$. Set $f(z)=\varphi(z)\sin(\varepsilon z)/z$. 
Then $fV\in L^2(\mathbb R)$, $\Type(fV)<\pi$, and, hence, $fV\in\Pw$, $f\in\mathcal H_V$. Thus, $\mathcal H_V\not=\{0\}$.

Now, $\mathcal H_V$ is a non-trivial de Branges space of entire functions. 
By \cite[Theorem 22]{deB} we can choose a real entire function $A$ with real zeros $(a_n)_{n\in\mathbb Z}$ such that 
$\mathcal Z(A)\cap\mathbb Z=\emptyset$, the system of the reproducing kernels of $\mathcal H_V$ at the points 
$a_n$, $n\in\mathbb Z$, or, equivalently, the system
$$
\Bigl(\frac{A}{\cdot-a_n}\Bigr)_{n\in\mathbb Z}
$$
are orthogonal bases in $\mathcal H_V$. Then $\Type (A)+\Type (V)=\pi$ otherwise, 
the function $z\mapsto A(z)\sin(\varepsilon z)/z$ would be in $\mathcal H_V$ and orthogonal to $\mathcal H_V$) and $A\in\Cart$
 (because $AV/(\cdot-a_n)\in\Pw$). 

Next, \eqref{d11} is equivalent to 
\begin{equation}
F\perp \frac{AV}{\cdot-a_n},\qquad n\in \mathbb Z. 
\label{d12}
\end{equation}

By the Parseval theorem, \eqref{d12} is equivalent to 
\begin{equation}
\sum_{k\in\mathbb Z}\frac{F(k)A(k)V(k)}{k-a_n}=0,\qquad n\in \mathbb Z.
\label{d15}
\end{equation}

Finally, \eqref{d15} implies that 
\begin{equation}
\sum_{k\in\mathbb Z}\frac{F(k)A(k)V(k)}{k-z}=\frac{\pi A(z)U(z)}{\sin \pi z},
\label{d16}
\end{equation}
for some entire function $U\in\Cart$. Furthermore, $\Type (A)+\Type (U)\le \pi$, and hence, $\Type (U)\le \Type (V)$. 
Comparing the residues on both sides in \eqref{d15}, we obtain \eqref{d10}.
\end{proof}

\begin{remark} \label{rem77} A similar argument using formula \eqref{d16} gives the following variant of the previous assertion. 
Let $F\in\Pw$, let $V\in\Cart$ be a real entire function with real simple zeros and with $\Type (V)<\pi$ such that  
$$
F\in\Spa_{\lambda\in\mathcal Z(V)}{\bf k}_\lambda.
$$
Then there exists $U\in\Cart$ such that $\Type(U)\le \Type(V)$ and 
\begin{align*}
F(n)&=(-1)^n\frac{U(n)}{V(n)},\qquad n\in\mathbb Z\setminus \mathcal Z(V),\\
U(n)&=0,\qquad n\in\mathbb Z\cap\mathcal Z(V).
\end{align*}
\end{remark}

An immediate application of Proposition~\ref{thm1c}, Lemmata~\ref{lem22}, \ref{lem28}, and Theorem~\ref{thm22} is the following result.

\begin{theorem}
There exists $\varepsilon>0$ 
such that if $\omega\in L^2(\mathbb R_+)$ is a decreasing function satisfying \eqref{d01} and \eqref{d17}, if $(\zeta_m)_{m\in\mathbb Z}$ is the sequence of independent standard Gaussian complex variables, and if $f\in L^2[-\pi,\pi]$ is such that 
$$
\widehat f(n)=\zeta_n\omega(|n|),\qquad n\in\mathbb Z,
$$
then, almost surely,  
$$
f\in\Spa\bigl\{e^{i\lambda t}: \lambda\in\mathcal Z(V)\bigr\},
$$
where $V\in\Cart$ is a real entire function with real zeros, and $D(\mathcal Z(V))=D_{BM}(\mathcal Z(V))<1-\varepsilon$. 
Moreover, the system $\{e^{i\lambda t}: \lambda\in\Lambda\}$ is a Riesz basis in its closed linear span in $L^2[-\pi,\pi]$, and 
$$
f(t)=\sum_{\lambda\in \mathcal Z(V)} a_\lambda e^{i\lambda t}
$$
with convergence in $L^2[-\pi,\pi]$, for some coefficients $(a_\lambda)_{\lambda\in\Lambda}\in\ell^2$.
\label{thm8}
\end{theorem}

\begin{proof}
We need to verify only the last statement. 
Let us recall that the zeros of $V$ are $\lambda_s=s+\delta_s$, $s\in U$, with $0\ne|\delta_s|\to 0$, $|s|\to\infty$, and 
$\lambda_{Ty}=Ty-\alpha_{Ty}$, $\lambda_{Ty+1}=Ty-\alpha_{Ty+1}$, $y\in S$, with $|\alpha_{Ty}+1/2|<1/4$, $|\alpha_{Ty+1}-1/2|<1/4$. 
Set $\lambda_{Ty}=Ty+3/2$, $y\in S$, and consider the system $\Lambda=(\lambda_n)_{n\in\mathbb Z}$. Then 
$\mathcal Z(V)\subset \Lambda$, $\Lambda\cap\mathbb Z=\emptyset$ and the points of $\Lambda$ are separated. 

A result of Avdonin \cite[Theorem 2]{Av} (an extension of the Kadec $1/4$ theorem) tells that if a system $\Lambda=(\lambda_n)_{n\in\mathbb Z}=(n+\delta_n)_{n\in\mathbb Z}$ is separated, $\Lambda\cap\mathbb Z=\emptyset$, and if for some $H\in\mathbb N$, $0<\delta<1/4$ we have
\begin{equation}
\Bigl | \sum_{kH\le n+\delta_n\le (k+1)H} \delta_n \Bigr|\le \delta H,
\label{avd}
\end{equation}
then the system $\{e^{i\lambda t}: \lambda\in\Lambda\}$ is a Riesz basis in $L^2[-\pi,\pi]$. Finally, for large $T$, say for $T\ge 100$ 
we can find large $H$ and $0<\delta<1/4$ such that our system $\Lambda$ satisfies condition \eqref{avd}. This completes the proof. 
\end{proof}

The assertion of Theorem~\ref{thm8} contains Theorem~\ref{t1} and gives more information on the set of frequencies we need to approximate $f$. 


  
\begin{remark} An alternative way to obtain the result of Theorem~\ref{thm8} would be to interpolate only the values $\widehat{f}(m)$, $m \in \mathbb{Z} \setminus U$, using the corresponding poles. After that, we could just complete the resulting set of exponents by the exponents with the set of frequencies $U$ using the Riesz basis property.  
\label{rem16dop}
\end{remark}
  
We complete our theorem by the following two remarks. 
The first of them tells that the Beurling--Malliavin density of the set of frequencies we need to approximate elements of $L^2[-\pi,\pi]$ cannot be smaller than $1/2$.
The second tells that there are $f\in L^2[-\pi,\pi]$ that cannot be approximated using any set of frequencies of Beurling--Malliavin 
density smaller than $1$.

\begin{remark}
Let $U,V,U_1,V_1$ be entire functions of exponential type such that $\Type (U)\le \Type (V)$, $\Type (U_1)\le \Type (V_1)$, $\bigl(\mathcal Z(V)\cup \mathcal Z(V_1)\bigr)\cap\mathbb Z=\emptyset$,
\begin{equation}
\frac{U(n)}{V(n)}=\delta_{0n}+\frac{U_1(n)}{V_1(n)},\qquad n\in\mathbb Z.
\label{d18}
\end{equation}
Then 
$$
\max\bigl(\Type (V),\Type (V_1)\bigr)\ge \frac\pi2.
$$
\label{rem9}
\end{remark}

\begin{proof} Suppose that $\max\bigl(\Type (U),\Type (U_1),\Type (V),\Type (V_1)\bigr)< \frac\pi2$. Let $F=U_1V-UV_1$. Then $zF$ is an entire functions of exponential type less than $\pi$ vanishing on $\mathbb Z$. Hence, $F=0$ which contradicts \eqref{d18}.
\end{proof}


\begin{remark} \label{rem10} There exist $F\in\Pw$ such that if $V\in\Cart$ is a real entire function with simple real zeros, and if $\Type(V)<\pi$, then 
$$
F\not\in\Spa_{\lambda\in\mathcal Z(V)}{\bf k}_\lambda.
$$
\end{remark} 

\begin{proof} 
Since $\Pw{\mid}_{\mathbb Z}=\ell^2(\mathbb Z)$, we can take $F\in\Pw$ such that 
$$
\limsup_{|n|\to\infty}\frac{\log F(n)}{|n|}< 0.
$$
Suppose that for some real entire function $V$ with simple real zeros and such that $\Type(V)<\pi$, we have 
$$
F\in\Spa_{\lambda\in\mathcal Z(V)}{\bf k}_\lambda.
$$
By Remark~\ref{rem77}, there exists $U\in\Cart$ such that $\Type(U)\le \Type(V)$ and 
\begin{align*}
F(n)&=(-1)^n\frac{U(n)}{V(n)},\qquad n\in\mathbb Z\setminus \mathcal Z(V),\\
U(n)&=0,\qquad n\in\mathbb Z\cap\mathcal Z(V).
\end{align*}
Then $U$ decays exponentially on $\mathbb Z$, and by Cartwright's theorem (\cite[Section 21.2]{L}), $U=0$ and then $F=0$.
\end{proof}

\subsection*{Acknowledgments} 
The authors are grateful to the referee for numerous pertinent remarks and suggestions. 

The first and the third authors are winners of the ``Leader'' competition conducted by the Foundation
for the Advancement of Theoretical Physics and Mathematics ``BASIS'' and would
like to thank its sponsors and jury.

\end{document}